\documentclass{amsart}
\usepackage{amsmath, amsfonts, amssymb}
\usepackage{xypic}
\usepackage{graphicx}

\newtheorem{thm1}{Theorem}
\newtheorem{thm}{Theorem}[section] 
\newtheorem{lemma}[thm]{Lemma}
\newtheorem{prop}[thm]{Proposition}
\newtheorem{cor}[thm]{Corollary}

\theoremstyle{definition}
\newtheorem{remark}[thm]{Remark}

\newtheorem{defn}[thm]{Definition}

\DeclareMathOperator{\Gr}{Gr}
\DeclareMathOperator{\Fl}{Fl}
\DeclareMathOperator{\LG}{LG}

\DeclareMathOperator{\OG}{OG}

\DeclareMathOperator{\Span}{Span}

\DeclareMathOperator{\QH}{QH}
\DeclareMathOperator{\QK}{QK}

\newcommand{\bP}{{\mathbb P}}
\renewcommand{\P}{{\mathbb P}}
\newcommand{\C}{{\mathbb C}}

\newcommand{\Z}{{\mathbb Z}}

\newcommand{\cO}{{\mathcal O}}

\newcommand{\cZ}{{\mathcal Z}}

\newcommand{\euler}[1]{\chi_{_{#1}}}
\newcommand{\pt}{\text{point}}

\newcommand{\ev}{\operatorname{ev}}
\newcommand{\wt}{\widetilde}

\newcommand{\wb}{\overline}

\newcommand{\pic}[2]{\includegraphics[scale=#1]{#2}}

\newcommand{\bd}{{\bf d}}
\newcommand{\be}{{\bf e}}
\newcommand{\ignore}[1]{}

\newcommand{\Mb}{\wb{\mathcal M}}
\def\noin{\noindent}

\def\N{{\mathbb N}}

\begin{document}

\title{Finiteness of cominuscule quantum $K$-theory}

\date{April 15, 2012}

\author{Anders~S.~Buch}
\address{Department of Mathematics, Rutgers University, 110
  Frelinghuysen Road, Piscataway, NJ 08854, USA}
\email{asbuch@math.rutgers.edu}

\author{Pierre--Emmanuel Chaput}
\address{Domaine Scientifique Victor Grignard, 239, Boulevard des
  Aiguillettes, Universit{\'e} Henri Poincar{\'e} Nancy 1, B.P. 70239,
  F-54506 Vandoeuvre-l{\`e}s-Nancy Cedex, France}
\email{chaput@iecn.u-nancy.fr}

\author{Leonardo~C.~Mihalcea}
\address{Department of Mathematics, Virginia Tech, 460 McBryde,
  Blacksburg VA 24060, USA}
\email{lmihalce@math.vt.edu}

\author{Nicolas Perrin}
\address{Hausdorff Center for Mathematics, Universit\"at Bonn, Villa
  Maria, Endenicher Allee 62, 53115 Bonn, Germany and Institut de
  Math{\'e}matiques de Jussieu, Universit{\'e} Pierre et Marie Curie,
  Case 247, 4 place Jussieu, 75252 Paris Cedex 05, France}
\email{nicolas.perrin@hcm.uni-bonn.de}

\subjclass[2000]{Primary 14N35; Secondary 19E08, 14N15, 14M15, 14M20, 14M22}

\thanks{The first author was supported in part by NSF Grant
  DMS-0906148.}

\begin{abstract}
  The product of two Schubert classes in the quantum $K$-theory ring
  of a homogeneous space $X = G/P$ is a formal power series with
  coefficients in the Grothendieck ring of algebraic vector bundles on
  $X$.  We show that if $X$ is cominuscule, then this power series has
  only finitely many non-zero terms.  The proof is based on a
  geometric study of boundary Gromov-Witten varieties in the
  Kontsevich moduli space, consisting of stable maps to $X$ that take
  the marked points to general Schubert varieties and whose domains
  are reducible curves of genus zero.  We show that all such varieties
  have rational singularities, and that boundary Gromov-Witten
  varieties defined by two Schubert varieties are either empty or
  unirational.  We also prove a relative Kleiman-Bertini theorem for
  rational singularities, which is of independent interest.  A key
  result is that when $X$ is cominuscule, all boundary Gromov-Witten
  varieties defined by three single points in $X$ are rationally
  connected.
\end{abstract}

\maketitle

\markboth{A.~BUCH, P.--E.~CHAPUT, L.~MIHALCEA, AND N.~PERRIN}
{FINITENESS OF COMINUSCULE QUANTUM $K$-THEORY}

\section{Introduction}\label{sec:intro}

The goal of this paper is to prove that any product of Schubert
classes in the quantum $K$-theory ring of a cominuscule homogeneous
space contains only finitely many non-zero terms.

Let $X = G/P$ be a homogeneous space defined by a semisimple complex
Lie group $G$ and a parabolic subgroup $P$, and let $\Mb_{0,n}(X,d)$
denote the Kontsevich moduli space of $n$-pointed stable maps to $X$
of degree $d$, with total evaluation map $\ev : \Mb_{0,n}(X,d) \to
X^n$.  Given Schubert varieties $\Omega_1,\dots,\Omega_n \subset X$ in
general position, there is a {\em Gromov-Witten variety\/}
$\ev^{-1}(\Omega_1 \times \dots \times \Omega_n) \subset
\Mb_{0,n}(X,d)$, consisting of all stable maps that send the $i$-th
marked point into $\Omega_i$ for each $i$.  The Kontsevich space and
its Gromov-Witten varieties are the foundation of the quantum
cohomology ring of $X$, whose structure constants are the
(cohomological) Gromov-Witten invariants, defined as the number of
points in finite Gromov-Witten varieties.  More generally, the {\em
  $K$-theoretic Gromov-Witten invariant\/}
$I_d(\cO_{\Omega_1},\dots,\cO_{\Omega_n})$ is defined as the sheaf
Euler characteristic of $\ev^{-1}(\Omega_1 \times \dots \times
\Omega_n)$, which makes sense when this variety has positive
dimension.  The $K$-theoretic invariants are more challenging to
compute, both because they are not enumerative, and also because they
do not vanish for large degrees.

Assume for simplicity that $P$ is a maximal parabolic subgroup of $G$,
so that $H_2(X;\Z) = \Z$.  The (small) {\em quantum $K$-theory ring\/}
$\QK(X)$ is a formal deformation of the Grothendieck ring $K(X)$ of
algebraic vector bundles on $X$, which as a group is defined by
$\QK(X) = K(X) \otimes_\Z \Z[[q]]$.  The product $\cO_u \star \cO_v$
of two Schubert structure sheaves is defined in terms of structure
constants $N^{w,d}_{u,v} \in \Z$ such that
\[
\cO_u \star \cO_v = \sum_{w,\,d \geq 0} N^{w,d}_{u,v}\, q^d\, \cO_w \,.
\]
In contrast to the quantum cohomology ring $\QH(X)$, the constants
$N^{w,d}_{u,v}$ are not single Gromov-Witten invariants, but are
defined as polynomial expressions of the $K$-theoretic Gromov-Witten
invariants.  A result of Givental asserts that $\QK(X)$ is an
associative ring \cite{givental:on}.  Since the $K$-theoretic
Gromov-Witten invariants do not vanish for large degrees, the same
might be true for the structure constants $N^{w,d}_{u,v}$, in which
case the product $\cO_u \star \cO_v$ would be a power series in $q$
with infinitely many non-zero terms.  When $X$ is a Grassmannian of
type A, a combinatorial argument in \cite{buch.mihalcea:quantum} shows
that this does not happen; all products in $\QK(X)$ are finite.  In
this paper we give a different geometric proof that shows more
generally that all products in $\QK(X)$ are finite whenever $X$ is a
cominuscule homogeneous space.  As a consequence, the quantum
$K$-theory ring $\QK(X)$ provides an honest deformation of $K(X)$.
The class of cominuscule varieties consists of Grassmannians of type
A, Lagrangian Grassmannians, maximal orthogonal Grassmannians, and
quadric hypersurfaces.  In addition there are two exceptional
varieties of type E called the Cayley plane and the Freudenthal
variety.

Let $d_X(n)$ be the minimal degree of a rational curve passing through
$n$ general points of $X$.  The numbers $d_X(n)$ for $n\leq 3$ have
been computed explicitly in \cite{chaput.manivel.ea:quantum*1,
  chaput.perrin:rationality}, see the table in \S\ref{sec:comin}
below.  Our main result is the following.

\begin{thm1}\label{thm:intro_vanish}
  Let $X$ be a cominuscule variety.  Then $N^{w,d}_{u,v} = 0$ for $d >
  d_X(2)$.
\end{thm1}

Theorem~\ref{thm:intro_vanish} holds also for the structure constants
of the equivariant quantum $K$-theory ring $\QK_T(X)$, see
Remark~\ref{rmk:eqkfinite}.  The bound on $d$ is sharp in the sense
that $q^{d_X(2)}$ occurs in the square of a point in $\QK(X)$.  In
addition, this bound is also the best possible for the quantum
cohomology ring $\QH(X)$ that does not depend on $u$, $v$, and $w$
(cf.\ \cite{fulton.woodward:on}).

Our proof uses that the structure constants $N^{w,d}_{u,v}$ can be
rephrased as alternating sums of certain boundary Gromov-Witten
invariants.  Given a sequence $\bd = (d_0,d_1,\dots,d_r)$ of effective
degrees $d_i \in H_2(X;\Z)$ such that $d_i>0$ for $i>0$ and $\sum d_i
= d$, let $M_\bd \subset \Mb_{0,3}(X,d)$ be the closure of the locus
of stable maps for which the domain is a chain of $r+1$ projective
lines that map to $X$ in the degrees given by $\bd$, the first and
second marked points belong to the first projective line, and the
third marked point is on the last projective line.  Then any constant
$N^{w,d}_{u,v}$ can be expressed as an alternating sum of sheaf Euler
characteristics of varieties of the form $\ev^{-1}(\Omega_1 \times
\Omega_2 \times \Omega_3) \cap M_\bd$.  We use geometric arguments to
show that the terms of this sum cancel pairwise whenever $X$ is
cominuscule and $d > d_X(2)$.

Set $\cZ_\bd = \ev(M_\bd) \subset X^3$.  A key technical fact in our
proof is that the general fibers of the map $\ev : M_\bd \to \cZ_\bd$
are rationally connected.  Notice that these fibers are boundary
Gromov-Witten varieties $\ev^{-1}(x \times y \times z) \cap M_\bd$
defined by three single points in $X$, and the result generalizes the
well known fact that there is a unique rational curve of degree $d$
through three general points in the Grassmannian $\Gr(d,2d)$
\cite{buch.kresch.ea:gromov-witten}.  In the special case when
$\bd=(d)$ and $M_\bd = \Mb_{0,3}(X,d)$, it was shown in
\cite{buch.mihalcea:quantum, chaput.perrin:rationality} that the
general fibers of $\ev$ are rational; our proof uses this case as well
as Graber, Harris, and Starr's criterion for rational connectivity
\cite{graber.harris.ea:families}.

We also need to know that $M_\bd$ has rational singularities.  For
this we prove a relative version of the Kleiman-Bertini theorem
\cite{kleiman:transversality} for rational singularities.  This
theorem implies that any boundary Gromov-Witten variety in
$\Mb_{0,n}(X,d)$ has rational singularities, for any homogeneous space
$X$.  The Kleiman-Bertini theorem generalizes a result of Brion
asserting that rational singularities are preserved when a subvariety
of a homogeneous space is intersected with a general Schubert variety
\cite{brion:positivity}.

Finally, if $\Omega_1$ and $\Omega_2$ are Schubert varieties in
general position in a homogeneous space $X$, we prove that
$\ev_1^{-1}(\Omega_1) \cap M_\bd$ is unirational and
$\ev_1^{-1}(\Omega_1) \cap \ev_2^{-1}(\Omega_2) \cap M_\bd$ is either
empty or unirational.  In particular, we have $I_d(\cO_{\Omega_1}) =
1$ and $I_d(\cO_{\Omega_1}, \cO_{\Omega_2}) \in \{0,1\}$.  This is
done by showing that any Borel-equivariant map to a Schubert variety
is locally trivial over the open cell.  In particular, any single
evaluation map $\ev_i : M_\bd \to X$ is locally trivial.

Our paper is organized as follows.  In section~\ref{sec:kb} we prove
the Kleiman-Bertini theorem for rational singularities and give a
simple criterion for an equivariant map to be locally trivial.  These
results are applied to (boundary) Gromov-Witten varieties of general
homogeneous spaces in section~\ref{sec:gwvar}.
Section~\ref{sec:comin} proves some useful facts about images of
Gromov-Witten varieties of cominuscule spaces, among them that the
general fibers of $\ev : M_\bd \to \cZ_\bd$ are rationally connected.
Finally, section~\ref{sec:qkcomin} applies these results to show that
$K$-theoretic quantum products on cominuscule varieties are finite.

Parts of this work was carried out during visits to the Mathematical
Sciences Research Institute (Berkeley), the Centre International de
Rencontres Math{\'e}matiques (Luminy), and the Max-Planck-Institut
f{\"u}r Mathematik (Bonn).  We thank all of these institutions for
their hospitality and stimulating environments.  We also benefited
from helpful discussions with P.~Belkale, S.~Kumar, and F.~Sottile.

\section{A Kleiman-Bertini theorem for rational singularities}\label{sec:kb}

\begin{defn}
  Let $G$ be a connected algebraic group and $X$ a $G$-variety.  A
  {\em splitting\/} of the action of $G$ on $X$ is a morphism $s : U
  \to G$ defined on a dense open subset $U \subset X$, together with a
  point $x_0 \in U$, such that $s(x).x_0 = x$ for all $x \in U$.  If a
  splitting exists, then we say that the action is {\em split\/} and
  that $X$ is $G$-{\em split}.
\end{defn}

Notice that any $G$-split variety contains a dense open orbit.
Schubert varieties are our main examples of varieties with a split
action.

\begin{prop}\label{prop:split}
  Let $G$ be a semisimple complex Lie group, $P \subset G$ a parabolic
  subgroup, and $X = G/P$ the corresponding homogeneous space with its
  natural $G$-action.  Then $X$ is $G$-split.  Furthermore, if $B
  \subset G$ is a Borel subgroup and $\Omega \subset X$ is a
  $B$-stable Schubert variety, then $\Omega$ is $B$-split.
\end{prop}
\begin{proof}
  Let $\Omega \subset X$ be a $B$-stable Schubert variety,
  $\Omega^\circ \subset \Omega$ the $B$-stable open cell, and $x_0 \in
  \Omega^\circ$ any point.  According to e.g.\ \cite[Lemma
  8.3.6]{springer:linear} we can choose a unipotent subgroup $U
  \subset B$ such that the map $U \to \Omega^\circ$ defined by $g
  \mapsto g.x_0$ is an isomorphism.  The inverse of this map is a
  splitting of the $B$-action on $\Omega$.  Since $X$ is a Schubert
  variety, it follows that $X$ is $B$-split and consequently
  $G$-split.
\end{proof}

Recall that a morphism $f : M \to X$ is a {\em locally trivial
  fibration\/} if each point $x \in X$ has an open neighborhood $U
\subset X$ such that $f^{-1}(U) \cong U \times f^{-1}(x)$ and $f$ is
the projection to the first factor.

\begin{prop}\label{prop:loctriv}
  Let $f : M \to X$ be an equivariant map of irreducible
  $G$-varieties.  Assume that $X$ is $G$-split.  Then $f$ is a locally
  trivial fibration over the dense open $G$-orbit in $X$, and the
  fibers over this orbit are irreducible.
\end{prop}
\begin{proof}
  Let $x_0 \in U \subset X$ and $s : U \to G$ be a splitting of the
  $G$-action on $X$.  Then the map $\varphi : U \times f^{-1}(x_0) \to
  f^{-1}(U)$ defined by $\varphi(x,y) = s(x).y$ is an isomorphism,
  with inverse given by $\varphi^{-1}(m) = (f(m), s(f(m))^{-1}.m)$.
  Since $f^{-1}(U) \cong U \times f^{-1}(x_0)$ is irreducible, so is
  $f^{-1}(x_0)$.
\end{proof}

In the rest of this section, a variety means a reduced scheme of
finite type over an algebraically closed field of characteristic zero.
An irreducible variety $X$ has {\em rational singularities\/} if there
exists a desingularization $\pi : \wt{X} \to X$ such that $\pi_*
\cO_{\wt{X}} = \cO_X$ and $R^i \pi_* \cO_{\wt{X}} = 0$ for all $i>0$.
An arbitrary variety has rational singularities if its irreducible
components have rational singularities, are disjoint, and have the
same dimension.  Zariski's main theorem implies that any variety with
rational singularities is normal.  Notice also that if $X$ and $Y$
have rational singularities, then so does $X \times Y$.  The converse
is a special case of the following lemma of Brion
\cite[Lemma~3]{brion:positivity}.

\begin{lemma}[Brion]\label{lem:ratsingfib}
  Let $Z$ and $S$ be varieties and let $\pi : Z \to S$ be a morphism.
  If $Z$ has rational singularities, then the same holds for the
  general fibers of $\pi$.
\end{lemma}

The following generalization of the Kleiman-Bertini theorem
\cite{kleiman:transversality} was proved by Brion in
\cite[Lemma~2]{brion:positivity} when $p$ and $q$ are inclusions and
$Y$ is a Schubert variety.  We adapt his proof to our case.

\begin{thm}\label{thm:KBratsings}
  Let $G$ be a connected algebraic group and let $X$ be a split and
  transitive $G$-variety.  Let $p : Y \to X$ and $q : Z \to X$ be
  morphisms of varieties, and assume that $Y$ and $Z$ have rational
  singularities.  Then $g.Y \times_X Z$ has rational singularities for
  all points $g$ in a dense open subset of $G$.
\end{thm}
\begin{proof}
  It follows from Proposition~\ref{prop:loctriv} that the map $m : G
  \times Y \to X$ defined by $m(g,y) = g.p(y)$ is a locally trivial
  fibration.  Set $Q = (G\times Y) \times_X Z$ and consider the
  diagram:
  \[ \xymatrix{ Q \ar[r] \ar[d] & G \times Y \ar[r]_{\text{pr}_1}
    \ar[d]^{m} & G \\
    Z \ar[r] & X }
  \]
  Since $G \times Y$ has rational singularities, it follows from
  Lemma~\ref{lem:ratsingfib} that $m^{-1}(x)$ has rational
  singularities for $x \in X$.  Since $m$ is a locally trivial
  fibration, so is the map $Q \to Z$, hence the assumption that $Z$
  has rational singularities implies that $Q$ has rational
  singularities.  Finally, Lemma~\ref{lem:ratsingfib} applied to the
  map $Q \to G$ implies that $g.Y \times_X Z$ has rational
  singularities for all points $g$ in a dense open subset of $G$.
\end{proof}

In the situation of Theorem~\ref{thm:KBratsings}, notice that if the
map $p : Y \to X$ is $G$-equivariant, then the isomorphism class of
$g.Y \times_X Z$ is independent of $g$.  It follows that $Y \times_X
Z$ has rational singularities.

\subsection{Rationality}

Before we continue, we recall some rationality properties of algebraic
varieties that are required in later sections.  An algebraic variety
$X$ is called {\em rational\/} if it is birationally equivalent to a
projective space $\bP^n$.  It is called {\em unirational\/} if there
exists a dominant morphism $U \to X$ where $U$ is an open subset of
$\bP^n$; here $n$ is allowed to be greater than the dimension of $X$,
but in such cases one can replace $\bP^n$ with a linear subspace to
obtain a generically finite map from $U$ to $X$.  Finally, $X$ is said
to be {\em rationally connected\/} if a general pair of points $(x,y)
\in X \times X$ can be joined by a rational curve, i.e.\ both $x$ and
$y$ belong to the image of some morphism $\bP^1 \to X$.  Rational
implies unirational, which in turn implies rational connectivity when
$X$ is complete.  Notice also that any rationally connected variety is
irreducible.  The relevance of these notions to our study of quantum
$K$-theory originates in the fact that, if $X$ is any rationally
connected non-singular projective variety, then $H^i(X,\cO_X) = 0$ for
all $i>0$ \cite[Cor.\ 4.18(a)]{debarre:higher-dimensional}, hence the
sheaf Euler characteristic of $X$ is equal to one.  In addition,
rational connectivity is one of the hypotheses needed in
Proposition~\ref{prop:gysin} below.  The following result from
\cite{graber.harris.ea:families} provides an important tool for
proving that a variety is rationally connected.

\begin{thm}[Graber, Harris, Starr]\label{thm:ratconn}
  Let $f : X \to Y$ be any dominant morphism of complete irreducible
  complex varieties.  If $Y$ and the general fiber of $f$ are
  rationally connected, then $X$ is rationally connected.
\end{thm}

\section{Geometry of Gromov-Witten varieties}\label{sec:gwvar}

Let $X = G/P$ be a homogeneous space, where $G$ is any semisimple
complex linear algebraic group and $P$ a parabolic subgroup.  Given an
effective class $d \in H_2(X;\Z)$ and an integer $n \geq 0$, the
Kontsevich moduli space $\Mb_{0,n}(X,d)$ parametrizes the set of all
$n$-pointed stable genus-zero maps $f : C \to X$ with $f_*[C] = d$,
and is equipped with a total evaluation map $\ev = (\ev_1,\dots,\ev_n)
: \Mb_{0,n}(X,d) \to X^n := X \times \dots \times X$.  A detailed
construction of this space can be found in the survey
\cite{fulton.pandharipande:notes}.  The space $\Mb_{0,n}(X,d)$ is a
projective variety with rational singularities, and it was proved by
Kim and Pandharipande that this variety is also rational
\cite{kim.pandharipande:connectedness}.

\begin{cor}\label{cor:gwratsing}
  Let $\Omega_1, \dots, \Omega_n \subset X$ be Schubert varieties of
  $X$ in general position.  Then the Gromov-Witten variety
  $\ev^{-1}(\Omega_1 \times \dots \times \Omega_n) \subset
  \Mb_{0,n}(X,d)$ has rational singularities.
\end{cor}
\begin{proof}
  Since the component-wise action of $G^n$ on $X^n$ is transitive and
  split, this result follows by applying Theorem~\ref{thm:KBratsings}
  to the inclusion $\Omega_1 \times \dots \times \Omega_n \subset X^n$
  and the evaluation map $\ev : \Mb_{0,n}(X,d) \to X^n$.
\end{proof}

The following proposition will be used to show that one-point and
two-point Gromov-Witten varieties are unirational.

\begin{prop}\label{prop:rat2gw}
  Let $M$ be a unirational $G$-variety and let $f_1 : M \to X$ and
  $f_2 : M \to X$ be $G$-equivariant maps.  Let $\Omega_1, \Omega_2
  \subset X$ be opposite Schubert varieties.\smallskip

  \noin{\em(a)} The variety $f_1^{-1}(\Omega_1) \subset M$ is
  unirational.\smallskip

  \noin{\em(b)} The image $\wt\Omega = f_2(f_1^{-1}(\Omega_1)) \subset
  X$ is a Schubert variety in $X$, and the map $f_2 :
  f_1^{-1}(\Omega_1) \to \wt\Omega$ is a locally trivial fibration
  over the open cell $\wt\Omega^\circ \subset \wt\Omega$.\smallskip

  \noin{\em(c)} The intersection $f_1^{-1}(\Omega_1) \cap
  f_2^{-1}(\Omega_2) \subset M$ is either empty or unirational.
\end{prop}
\begin{proof}
  It follows from Proposition~\ref{prop:loctriv} that $f_1$ is a
  locally trivial fibration, so the assumption that $M$ is unirational
  implies that the fibers of $f_1$ are unirational.  Part (a) follows
  from this because $f_1 : f_1^{-1}(\Omega_1) \to \Omega_1$ is also
  locally trivial.  Choose a Borel subgroup $B \subset G$ such that
  $\Omega_1$ is $B$-stable and $\Omega_2$ is $B^{\text{op}}$-stable.
  Then $f_1^{-1}(\Omega_1)$ and $\wt\Omega$ are $B$-stable, and since
  $\wt\Omega$ is also closed in $X$ and irreducible, it is a Schubert
  variety.  Part (b) now follows from Proposition~\ref{prop:loctriv}
  because $f_2 : f_1^{-1}(\Omega_1) \to \wt\Omega$ is $B$-equivariant.
  Parts (a) and (b) imply that the fibers of $f_2 : f_1^{-1}(\Omega_1)
  \to \wt\Omega$ over $\wt\Omega^\circ$ are unirational.  Since
  $\wt\Omega$ is normal, it follows using Stein factorization that all
  fibers of $f_2 : f_1^{-1}(\Omega_1) \to \wt\Omega$ are connected.
  If $f_1^{-1}(\Omega_1) \cap f_2^{-1}(\Omega_2) \neq \emptyset$, then
  the Kleiman-Bertini theorem \cite[Rmk.~7]{kleiman:transversality}
  implies that this intersection is locally irreducible.  Part (c)
  follows from this, using that the Richardson variety $\wt\Omega \cap
  \Omega_2$ is rational.
\end{proof}

\begin{cor}\label{cor:gw2unirat}
  {\em(a)} Let $\Omega \subset X$ be a Schubert variety and $n \geq
  1$.  Then $\ev_1^{-1}(\Omega) \subset \Mb_{0,n}(X,d)$ is
  unirational.

  {\em(b)} Let $\Omega_1, \Omega_2 \subset X$ be opposite Schubert
  varieties and $n\geq 2$.  Then the two-point Gromov-Witten variety
  $\ev_1^{-1}(\Omega_1) \cap \ev_2^{-1}(\Omega_2) \subset
  \Mb_{0,n}(X,d)$ is either empty or unirational.
\end{cor}
\begin{proof}
  This follows from parts (a) and (c) of
  Proposition~\ref{prop:rat2gw}.
\end{proof}

\begin{remark}
  Proposition~\ref{prop:loctriv} can also be used to prove
  unirationality of certain $3$-point Gromov-Witten varieties.  A
  result of Popov \cite{popov:generically} shows that the diagonal
  action of $G$ on $X^3$ has a dense open orbit if and only if $X$ is
  a cominuscule variety, and it is natural to ask if the action is
  also split.  It turns out that a splitting can be constructed when
  $X$ is a Grassmann variety of type A or a maximal orthogonal
  Grassmannian, but no splitting exists for a Lagrangian Grassmannian.
  When $X^3$ is $G$-split, it follows from
  Proposition~\ref{prop:loctriv} that $\ev^{-1}(x,y,z) \subset
  \Mb_{0,3}(X,d)$ is either empty or unirational for all points
  $(x,y,z)$ in the dense open $G$-orbit of $X^3$.  This partially
  recovers results from \cite{buch.mihalcea:quantum,
    chaput.perrin:rationality} asserting that 3-point Gromov-Witten
  varieties are rational for all cominuscule varieties (see
  Theorem~\ref{thm:rat3gw} below).  It is interesting to note that
  Lagrangian Grassmannians also required special treatment in
  \cite{chaput.perrin:rationality}.
\end{remark}

\begin{remark}
  Jason Starr reports that the results of
  \cite{jong.starr.ea:families} can be used to prove the following
  statement.  If $P \subset G$ is a maximal parabolic subgroup and $d$
  is sufficiently large, then $\ev^{-1}(x,y,z) \subset \Mb_{0,3}(X,d)$
  is rationally connected for all points $(x,y,z)$ in a dense open
  subset of $X^3$.
\end{remark}

Our applications require generalizations of Corollaries
\ref{cor:gwratsing} and \ref{cor:gw2unirat} to Gromov-Witten varieties
of stable maps with reducible domains.  Let $\bd =
(d_0,d_1,\dots,d_r)$ be a sequence of effective classes $d_i \in
H_2(X;\Z)$, let $\be = (e_0,\dots,e_r) \in \N^{r+1}$, and set
$|\bd|=\sum d_i$ and $|\be|=\sum e_i$.  We consider stable maps $f : C
\to X$ in $\Mb_{0,|\be|}(X,|\bd|)$ defined on a chain $C$ of $r+1$
projective lines, such that the $i$-th projective line contains $e_i$
marked points (numbered from $1+\sum_{j<i}e_i$ to $\sum_{j\leq i}e_i$)
and the restriction of $f$ to this component has degree $d_i$.  To
ensure that such a map is indeed stable, we demand that $e_i \geq 1 +
\delta_{i,0} + \delta_{i,r}$ whenever $d_i=0$.

Let $M_{\bd,\be} \subset \Mb_{0,|\be|}(X,|\bd|)$ be the closure of the
locus of all such stable maps.  This variety can also be defined
inductively as follows.  If $r=0$ then set $M_{d_0,e_0} =
\Mb_{0,e_0}(X,d_0)$.  Otherwise set $\bd' = (d_0, \dots, d_{r-1})$ and
$\be' = (e_0, \dots, e_{r-2}, e_{r-1}+1)$, and define $M_{\bd,\be}$ as
the product over $X$ of the maps $\ev_{|\be'|} : M_{\bd',\be'} \to X$
and $\ev_1 : M_{d_r,e_r+1} \to X$.  Given subvarieties $\Omega_1,
\dots, \Omega_m$ of $X$ with $m \leq |\be|$, define a boundary
Gromov-Witten variety by $M_{\bd,\be}(\Omega_1,\dots,\Omega_m) =
\bigcap_{i=1}^m \ev_i^{-1}(\Omega_i) \subset M_{\bd,\be}$.  The
varieties $\Omega_i$ will often, but not always, be chosen in general
position.

We also define the varieties $\cZ_{\bd,\be}(\Omega_1,\dots,\Omega_m) =
\ev(M_{\bd,\be}(\Omega_1,\dots,\Omega_m)) \subset X^{|\be|}$ and
$\Gamma_{\bd,\be}(\Omega_1,\dots,\Omega_m) =
\ev_{|\be|}(M_{\bd,\be}(\Omega_1,\dots,\Omega_m)) \subset X$.  If no
sequence $\be$ is specified, we will use $\be = (3)$ when $r=0$ and
$\be = (2,0,\dots,0,1)$ when $r>0$; this convention will only be used
when $d_i \neq 0$ for $i>0$.  For example, if $x,y \in X$ and $d \in
H_2(X;\Z)$, then $\Gamma_d(x) \subset X$ is the union of all rational
curves of degree $d$ passing through $x$, and $\Gamma_d(x,y)$ is the
union of all rational curves of degree $d$ passing through $x$ and
$y$.  The variety $\cZ_{d,2} \subset X \times X$ contains all pairs of
points that are connected by a rational curve of degree $d$, and
$\cZ_d = \cZ_{d,3} \subset X \times X \times X$ consists of the
triples of points connected by such a curve.

If $\Omega \subset X$ is a $B$-stable Schubert variety, then so is
$\Gamma_\bd(\Omega)$ by Proposition~\ref{prop:rat2gw}(b) and the
following result.  This fact was also used in
\cite{chaput.perrin:on}.

\begin{prop}\label{prop:ratsingMde}
  The variety $M_{\bd,\be}$ is unirational and has rational
  singularities.
\end{prop}
\begin{proof}
  By induction on $r$ we may assume that $M_{\bd',\be'}$ is
  unirational and has rational singularities.  Since all maps in the
  Cartesian square
  \[
  \xymatrix{
    M_{\bd,\be} \ar[r]^{p\ \ \ } \ar[d]^{q} & M_{d_r,e_r+1} \ar[d]^{\ev_1} \\
    M_{\bd',\be'} \ar[r]^{\ev_{|\be'|}} & X }
  \]
  are equivariant, it follows from Theorem~\ref{thm:KBratsings} that
  $M_{\bd,\be}$ has rational singularities.
  Proposition~\ref{prop:loctriv} implies that $\ev_1$ is locally
  trivial, and since $M_{d_r,e_r+1}$ is rational, we deduce that the
  fibers of $\ev_1$ are unirational.  This in turn implies that $q$ is
  a locally trivial map with unirational fibers.  Finally, since
  $M_{\bd',\be'}$ is unirational, we conclude that $M_{\bd,\be}$ is
  unirational as well.
\end{proof}

\begin{cor}\label{cor:unirat}
  Let $\Omega_1, \Omega_2, \dots, \Omega_m \subset X$ be Schubert
  varieties in general position, with $m \leq |\be|$.  Then
  $M_{\bd,\be}(\Omega_1, \dots, \Omega_m)$ has rational singularities.
  Furthermore, the one-point Gromov-Witten variety
  $M_{\bd,\be}(\Omega_1)$ is unirational, and the two-point
  Gromov-Witten variety $M_{\bd,\be}(\Omega_1,\Omega_2)$ is either
  empty or unirational.
\end{cor}

Let $d \in H_2(X;\Z)$ be an effective class and let $\pi_1, \pi_2 :
\cZ_{d,2} \to X$ denote the projections.  Define the variety
\[
\cZ^*_{d,2} = \cZ_{d,2} \setminus \bigcup_{d'} \cZ_{d',2}
\]
where the union is over all degrees $d' \in H_2(X;\Z)$ for which
$\cZ_{d',2} \varsubsetneq \cZ_{d,2}$.  This is a $G$-stable dense open
subset of $\cZ_{d,2}$ because $\cZ_{d,2} = \ev(M_{d,2})$ is
irreducible.  For $x \in X$ we also set $\Gamma^*_d(x) =
\pi_2(\pi_1^{-1}(x) \cap \cZ^*_{d,2})$, a dense open subset of
$\Gamma_d(x)$.

\begin{lemma}\label{lem:ratint2}
  Let $d \in H_2(X;\Z)$ be an effective class and $\Omega \subset X$ a
  Schubert variety.
  
  \noin{\rm(a)} The variety $\cZ_{d,2}$ is rational and has rational
  singularities.
 
  \noin{\rm(b)} The intersection $\Omega \cap \Gamma_d(z)$ is
  unirational and has rational singularities for all points $z$ in the
  open cell $\Gamma_d(\Omega)^\circ \subset \Gamma_d(\Omega)$.
  
  \noin{\rm(c)} Let $\Omega^* \subset \Omega$ be any dense open
  subset.  Then $\Omega^* \cap \Gamma_d^*(z) \neq \emptyset$ for all
  points $z$ in a dense open subset of $\Gamma_d(\Omega)$.
\end{lemma}
\begin{proof}
  It follows from Proposition~\ref{prop:loctriv} that the projection
  $\pi_1 : \cZ_{d,2} \to X$ is a locally trivial fibration.  Since
  each fiber $\pi_1^{-1}(x) \cong \Gamma_d(x)$ is a Schubert variety
  in $X$, we deduce that $\pi_1^{-1}(\Omega)$ is rational and has
  rational singularities.  Part (a) follows as a special case of this.
  Since Proposition~\ref{prop:rat2gw}(b) implies that $\pi_2 :
  \pi_1^{-1}(\Omega) \to \Gamma_d(\Omega)$ is a locally trivial
  fibration over $\Gamma_d(\Omega)^\circ$, it follows that
  $\pi_1^{-1}(\Omega) \cap \pi_2^{-1}(z) \cong \Omega \cap
  \Gamma_d(z)$ is unirational for all $z \in \Gamma_d(\Omega)^\circ$,
  and Lemma~\ref{lem:ratsingfib} implies that $\Omega \cap
  \Gamma_d(z)$ has rational singularities.  This proves (b).  Since
  $\pi_1(\cZ_{d,2}^*) = X$ and $\pi_1^{-1}(\Omega)$ is irreducible, we
  deduce that $U = \pi_1^{-1}(\Omega^*) \cap \cZ_{d,2}^*$ is a dense
  open subset of $\pi_1^{-1}(\Omega)$.  Part (c) follows from this
  because $\pi_2(U)$ contains a dense open subset of
  $\Gamma_d(\Omega)$, and $\Omega^* \cap \Gamma_d^*(z) \cong U \cap
  \pi_2^{-1}(z) \neq \emptyset$ for all $z \in \pi_2(U)$.
\end{proof}

\section{Gromov-Witten varieties of cominuscule
  spaces}\label{sec:comin}

Let $X = G/P$ be a homogeneous space, where $G$ is a simple complex
linear algebraic group and $P$ is a {\em maximal\/} parabolic
subgroup.  Fix a maximal torus $T$ and a Borel subgroup $B$ such that
$T \subset B \subset P \subset G$, and let $R$ be the associated root
system, with positive roots $R^+ \subset R$ and simple roots $\Delta
\subset R^+$.  Let $W = N_G(T)/T$ be the Weyl group of $G$ and $W_P =
N_P(T)/T \subset W$ the Weyl group of $P$.  The subgroup $P$
corresponds to a simple root $\alpha \in \Delta$, such that $W_P$ is
generated by all simple reflections except $s_\alpha$.  The variety
$X$ is called {\em cominuscule\/} if $\alpha$ is a cominuscule simple
root, i.e.\ when the highest root of $R$ is expressed as a linear
combination of simple roots, the coefficient of $\alpha$ is one.  The
collection of cominuscule varieties include the type A Grassmannians
$\Gr(m,n)$, Lagrangian Grassmannians $\LG(n,2n)$, maximal orthogonal
Grassmannians $\OG(n,2n)$, quadric hypersurfaces $Q^n \subset
\bP^{n+1}$, as well as two exceptional varieties called the Cayley
Plane ($E_6/P_6$) and the Freudenthal variety ($E_7/P_7$).  We will
assume that $X$ is cominuscule in the following.

Each element $u \in W$ defines a $T$-fixed point $u.P \in X$ and a
{\em Schubert variety} $X(u) = \overline{B u.P} \subset X$.  Both
$u.P$ and $X(u)$ depend only on the coset of $u$ in $W/W_P$. Let $W^P
\subset W$ be the set of minimal length representatives for the cosets
in $W/W_P$.  Then $W^P$ is in one-to-one correspondence with the set
of $T$-fixed points in $X$ as well as the set of $B$-stable Schubert
varieties in $X$, and for $u \in W^P$ we have $\dim X(u) = \ell(u)$.
We will identify $H_2(X;\Z) = \Z\, [X(s_\alpha)]$ with the group of
integers.  The {\em degree\/} of a curve $C \subset X$ is the integer
$d \in \N$ for which $[C] = d\,[X(s_\alpha)]$.

Given two points $x,y \in X$, we let $d(x,y)$ denote the smallest
degree of a rational curve containing $x$ and $y$ \cite{zak:tangents}.
Equivalently, $d(x,y)$ is minimal with the property that $(x,y) \in
\cZ_{d(x,y),2}$.  For any $n \in \N$ we also let $d_X(n)$ be the
smallest degree for which any collection of $n$ points in $X$ is
contained in a connected rational curve of degree $d_X(n)$, i.e.\
$d_X(n)$ is minimal such that $\cZ_{d_X(n),n} = X^n$.  Notice that
$d_X(2) = \max \{ d(x,y) \mid x,y\in X \}$, and it follows from
\cite{fulton.woodward:on} that $d_X(2)$ is the smallest degree of the
quantum parameter $q$ that occurs in the square of a point in the
small quantum ring $\QH(X)$.  Furthermore, for any degree $d \in \N$
we have $\cZ^*_{d,2} = \{ (x,y) \in X^2 \mid d(x,y)=d' \}$ and
$\Gamma^*_d(x) = \{ y \in X \mid d(x,y) = d' \}$, where $d' =
\min(d,d_X(2))$.  The numbers $d_X(2)$ and $d_X(3)$ were computed in
\cite[Prop.~18]{chaput.manivel.ea:quantum*1} and
\cite[Prop.~3.4]{chaput.perrin:rationality} and are given by the
following table.\medskip

\begin{center}
\begin{tabular}{c c c c c c}
  \hline\vspace{-3mm}\\
  $X$ & $\dim(X)$ & $d_X(2)$ & $d_X(3)$\vspace{1mm}\\
  \hline\vspace{-3mm}\\
  $\Gr(m,m+k)$ & $mk$ & $\min(m,k)$ & $\min(2m,2k,\max(m,k))$\vspace{1mm}\\
  $\LG(n,2n)$ & $\frac{n(n+1)}{2}$ & $n$ & $n$\vspace{1mm}\\
  $\OG(n,2n)$ & $\frac{n(n-1)}{2}$ & $\lfloor\frac{n}{2}\rfloor$ &
  $\lceil\frac{n}{2}\rceil$\vspace{1mm}\\
  $Q^n$ & $n$ & $2$ & $2$\vspace{1mm}\\
  $E_6/P_6$ & 18 & 2 & 4\vspace{1mm}\\
  $E_7/P_7$ & 27 & 3 & 3\vspace{1mm}\\
  \hline
\end{tabular} 
\end{center}
\bigskip

We require the following proposition, which combines parts of Prop.~18
and Lemma 21 from \cite{chaput.manivel.ea:quantum*1}.  Notice that part
(c) implies that $d_X(2)$ is equal to the number of occurrences of
$s_\alpha$ in a reduced word for the element $u \in W^P$ for which
$X(u)=X$.

\begin{prop}[\cite{chaput.manivel.ea:quantum*1}]\label{prop:cmpfacts}
  Let $X = G/P$ be a cominuscule variety.\smallskip
  
  \noin{\em(a)} The diagonal action of $G$ on $\cZ^*_{d,2}$ is
  transitive for each $d \in [0,d_X(2)]$.\smallskip

  \noin{\em(b)} Let $x,y\in X$ and set $d = d(x,y)$.  Then the
  stabilizer in $G$ of the subvariety $\Gamma_d(x,y) \subset X$ is a
  parabolic subgroup of $G$ that acts transitively on
  $\Gamma_d(x,y)$.\smallskip

  \noin{\em(c)} Let $u \in W^P$.  Then $d(1.P, u.P)$ is the number of
  occurrences of $s_\alpha$ in any reduced expression for $u$.
\end{prop}

Proposition~\ref{prop:cmpfacts} is the foundation of the following
construction from \cite{chaput.manivel.ea:quantum*1}.  Fix a degree $d
\in [0,d_X(2)]$ and let $Y_d$ be the set of all subvarieties
$\Gamma_d(x,y)$ of $X$ for which $d(x,y)=d$.  The group $G$ acts on
$Y_d$ by translation, and by parts (a) and (b) of
Proposition~\ref{prop:cmpfacts} we can identify $Y_d$ with a
projective homogeneous space for $G$.  The points of the variety $Y_d$
provide a generalization to cominuscule varieties of the kernel-span
pairs of curves in classical Grassmannians \cite{buch:quantum,
  buch.kresch.ea:gromov-witten}.  For example, when $X = \Gr(m,n)$ is
a Grassmannian of type A, the space $Y_d$ is the two-step flag variety
$\Fl(m-d,m+d;n)$ of kernel-span pairs of expected dimension.  Notice
that the homogeneous space $Y_d$ contains a unique $B$-fixed point,
which implies that there is a unique $B$-stable Schubert variety in
$X$ of the form $\Gamma_d(x,y)$.  We denote this non-singular Schubert
variety by $X_d$.

\begin{lemma}\label{lem:linechain}
  Let $x, y \in X$.  Then there exists a chain of $d(x,y)$ rational
  curves of degree $1$ through $x$ and $y$.
\end{lemma}
\begin{proof}
  We may assume that $x = 1.P$ and $y = u.P$ are $T$-fixed points,
  with $u \in W^P$.  Let $u = s_{\beta_1} s_{\beta_2} \cdots
  s_{\beta_\ell}$ be a reduced expression for $u$, and let $j_1 < j_2
  < \dots < j_{d(x,y)}$ be the indices for which $\beta_{j_i} =
  \alpha$.  Set $u_0 = 1$ and $u_i = s_{\beta_1} s_{\beta_2} \cdots
  s_{\beta_{j_i}}$ for each $i \in [1,d(x,y)]$.  Then $u_0.P = x$, and
  since $u \in W^P$, we must have $u = u_{d(x,y)}$, so $u_{d(x,y)}.P =
  y$.  Finally notice that $d(u_{i-1}.P, u_i.P) = d(1.P, u_{i-1}^{-1}
  u_i.P) = 1$, so $u_{i-1}.P$ and $u_i.P$ can be joined by a line in
  $X$.
\end{proof}

\begin{remark}
  The assumption that $X$ is cominuscule is necessary in
  Lemma~\ref{lem:linechain}.  For example, if $X = \OG(2,7)$ is the
  orthogonal Grassmannian of 2-dimensional isotropic subspaces in
  $\C^7$ equipped with a non-degenerate symmetric bilinear form, then
  two general points in $X$ are joined by an irreducible curve of
  degree 2, but not by a union of two lines.  The same is true if $X$
  is any adjoint variety \cite{chaput.perrin:on}.
\end{remark}

It follows from Lemma~\ref{lem:linechain} that if $\bd =
(d_0,\dots,d_r) \in \N^{r+1}$ is any sequence with $d_i>0$ for $i>0$,
and $\Omega \subset X$ is any closed subvariety, then 
$\Gamma_\bd(\Omega) = \Gamma_{|\bd|}(\Omega)$.  The following result
shows how to find the degree 1 neighborhood of a Schubert variety in
$X$.  Let $w_P$ denote the unique longest element in $W_P$.

\begin{lemma}\label{lem:deg1nbhd}
  Let $u \in W^P$ be such that $X(u) \neq X$.  Then $\Gamma_1(X(u)) =
  X(u w_P s_\alpha)$.
\end{lemma}
\begin{proof}
  Notice that $\Gamma_1(X(u))$ is a $B$-stable Schubert variety in
  $X$.  The inclusion $X(u w_P s_\alpha) \subset \Gamma_1(X(u))$
  follows because $d(u.P, u w_P s_\alpha.P) = d(1.P, s_\alpha.P) = 1$.
  To prove the opposite inclusion it is enough to show that any
  $T$-fixed point $w.P \in \Gamma_1(X(u))$ is contained in $X(u w_P
  s_\alpha)$.  Since $T$ acts on the projective variety of all degree
  one curves from $w.P$ to $X(u)$, there exists a $T$-stable curve of
  this kind, hence we have $d(w.P, v.P) = 1$ for some fixed point $v.P
  \in X(u)$.  We may assume that $v, w \in W^P$, so that $v \leq u$.
  Proposition~\ref{prop:cmpfacts}(c) implies that we can write $v^{-1}
  w = x s_\alpha y$ where $x,y \in W_P$.  Since the assumption $X(u)
  \neq X$ implies that $u w_P \leq u w_P s_\alpha$, we obtain $w
  y^{-1} = v x s_\alpha \leq u w_P s_\alpha$, so $w.P \in X(u w_P
  s_\alpha)$ as required.
\end{proof}

\begin{prop}\label{prop:dx3xd}
  Let $d \in [0,d_X(2)]$.  Then we have $\Gamma_{d_X(3)-d}(X_d) = X$.
\end{prop}
\begin{proof}
  We check the truth of this statement case by case.  Assume first
  that $X = \Gr(m,n)$ is the type A Grassmannian of all
  $m$-dimensional subspaces in $\C^n$.  Set $k = n-m$.  Since
  $\Gr(m,n) \cong \Gr(k,n)$, we may assume that $m \leq k$.  It
  follows from \cite[Lemma 1]{buch:quantum} and
  \cite[Prop.~1]{buch.kresch.ea:gromov-witten} that for $x,y \in X$ we
  have $d(x,y) = \dim(x + y) - m$, where $x+y = \Span(x,y) \subset
  \C^n$.  It follows that $d_X(2) = m$.  We furthermore have $X_d =
  \Gr(d,B/A) = \{ x \in X \mid A \subset x \subset B \}$ for some
  $(A,B) \in \Fl(m-d,m+d;n)$.  We claim that $\Gamma_{k-d}(X_d) = X$.
  Let $y \in X$ be any point and notice that $\dim(B \cap y) \geq
  2m+d-n = m+d-k$.  Since $m+d-k \leq d$, there exists a point $x \in
  X_d$ such that $\dim(x \cap y) \geq m+d-k$, or equivalently $d(x,y)
  \leq k-d$, as required.  We finally prove that $d_X(3) =
  \min(k,2m)$.  The inequality $d_X(3) \geq \min(k,2m)$ follows from
  \cite[Lemma~1]{buch:quantum}, and the opposite inequality follows
  from the claim and the observation that $\Gamma_{2m-d}(X_d) \supset
  \Gamma_m(\pt) = X$.  The identity $\Gamma_{d_X(3)-d}(X_d) = X$
  follows.

  Next assume that $X = \LG(n,2n)$ is the Lagrangian Grassmannian of
  maximal isotropic subspaces of a symplectic vector space $\C^{2n}$.
  For $x,y \in X$ we have $d(x,y) = \dim(x+y)-n$, which implies that
  $d_X(2) = n$, and $X_d = \LG(A^\perp/A) \subset X$ for some
  isotropic $A \subset \C^{2n}$ with $\dim(A) = n-d$.  In particular,
  we have $X_n = X$, so $d_X(3) = n$.  Let $y \in X$ be any point and
  set $x = (y \cap A^\perp) + A$.  Then $x \subset \C^{2n}$ is
  isotropic.  Write $\dim(y+x) = \dim(y+A) = n+t$.  Then $\dim(y \cap
  A^\perp) = \dim((y+A)^\perp) = n-t$ and $\dim(y \cap A) = n-t-d$.
  It follows that $\dim(x) = (n-t)+(n-d)-(n-t-d) = n$, so $x \in X_d$
  and $d(x,y) = t \leq n-d$, as required.
  
  Let $X = \OG(n,2n)$ be an orthogonal Grassmannian.  Given an
  orthogonal form on $\C^{2n}$ and a fixed maximal isotropic subspace
  $x_0 \subset \C^{2n}$, this space consists of all maximal isotropic
  subspaces $x \subset \C^{2n}$ such that $\dim(x+x_0)-n$ is even.  We
  have $d(x,y) = \frac{1}{2}(\dim(x+y)-n)$ for $x,y \in X$, hence
  $d_X(2) = \lfloor \frac{n}{2} \rfloor$, and $X_d = \OG(2d,A^\perp/A)
  \subset X$ for some isotropic subspace $A \subset \C^{2n}$ of
  dimension $n-2d$.  We claim that $d_X(3) = \lceil \frac{n}{2}
  \rceil$ and $\Gamma_{d_X(3)-d}(X_d) = X$.  Given any point $y \in
  X$, set $x = (y \cap A^\perp) + A$ and write $\dim(x+y) = n+t$.
  Then $\dim(x) = n$ and $t \leq n-2d$.  If $t$ is even, then $x \in
  X_d$ and $d(x,y) = \frac{t}{2} \leq \lfloor \frac{n}{2} \rfloor -
  d$.  Otherwise we can find a point $z \in X_d$ such that
  $\dim(x+z)=n+1$, and this implies that $d(z,y) = \frac{t+1}{2} \leq
  \lceil \frac{n}{2} \rceil - d$, proving the claim.
  
  Let $X = Q^n \subset \P(\C^{n+2})$ be a quadric hypersurface
  consisting of all isotropic lines through the origin in the vector
  space $\C^{n+2}$ equipped with an orthogonal form.  Then
  \cite[Prop.~18]{chaput.manivel.ea:quantum*1} shows that $d_X(2) = 2$
  and $X_2 = X$.  It follows that $d_X(3)=2$.  We must show that
  $\Gamma_1(X_1) = X$.  We have $X_1 = \bP(V)$ for some 2-dimensional
  isotropic subspace $V \subset \C^{n+2}$.  Given any point $y \in X$,
  choose a 1-dimensional subspace $x \subset V \cap y^\perp$.  Then $x
  \in X_1$, and since $x$ and $y$ are joined by the line $\bP(x+y)
  \subset Q^n$ we have $d(x,y)=1$, as required.
  
  Let $X = E_7/P_7$ be the Freudenthal variety.  In other words, $G$
  has type $E_7$ and $\alpha$ is the 7-th simple root of the Dynkin
  diagram:
  \[
  \pic{.4}{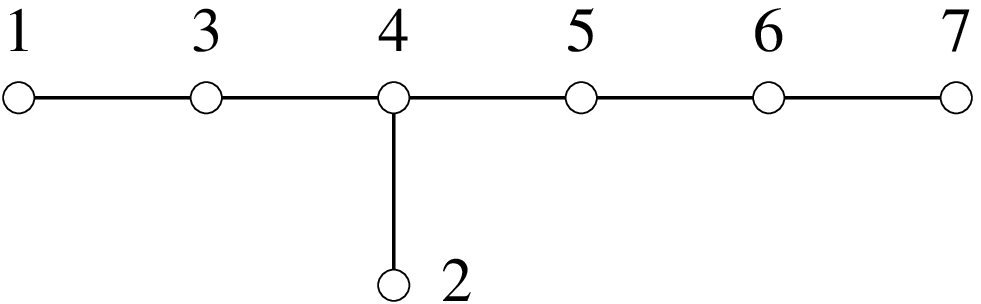}
  \]

  According to \cite[Prop.~18]{chaput.manivel.ea:quantum*1} we have
  $d_X(2) = 3$ and $X_3 = X$, which implies that $d_X(3) = 3$.  The
  description of the varieties $X_d$ in terms of quivers given in
  \cite{chaput.manivel.ea:quantum*1} also reveals that $X_2 =
  X(s_7s_6s_5s_4s_2s_3s_4s_5s_6s_7)$.  (Alternatively, the variety
  $X_2 \cong Q^{10}$ can be identified as the unique non-singular
  Schubert variety in $X$ of dimension 10.)  By
  Lemma~\ref{lem:deg1nbhd} we now obtain $\Gamma_1(X_2) = X$.  We also
  have $X_1 = X(s_7)$, $\Gamma_1(X_1) =
  X(s_1s_3s_4s_2s_5s_4s_3s_1s_7s_6s_5s_4s_2s_3s_4s_5s_6s_7)$, and
  $\Gamma_1(\Gamma_1(X_1)) = X$.
  
  Finally, let $X = E_6/P_6$ be the Cayley plane, i.e.\ $G$ has type
  $E_6$ and $\alpha$ is the 6-th simple root of the Dynkin diagram
  obtained by discarding node 7 in the above diagram.  By
  \cite[Prop.~18]{chaput.manivel.ea:quantum*1} we have $d_X(2)=2$ and
  \cite[Lemma~2.14]{chaput.perrin:rationality} shows that $d_X(3)=4$.
  We obtain $d_X(3)-d \geq 2$, so $\Gamma_{d_X(3)-d}(X_d) \supset
  \Gamma_2(\pt) = X$.
\end{proof}

\begin{cor}\label{cor:gammaschubert}
  Let $x,y \in X$ and let $d \geq d(x,y)$.  Then we have
  $\Gamma_d(x,y) = \Gamma_{d-d(x,y)}(\Gamma_{d(x,y)}(x,y))$.  In
  particular, $\Gamma_d(x,y)$ is a Schubert variety in $X$.
\end{cor}
\begin{proof}
  Let $z \in \Gamma_d(x,y)$.  We must show that $z \in
  \Gamma_{d-d_0}(\Gamma_{d_0}(x,y))$, where $d_0 = d(x,y)$.  If $d
  \geq d_X(3)$, then this follows from Proposition~\ref{prop:dx3xd}.
  On the other hand, if $d \leq d_X(2)$, then
  \cite[Prop.~19]{chaput.manivel.ea:quantum*1} implies that that $x,y,z$
  are contained in a translate of $X_d$.  It follows from
  \cite[Prop.~18]{chaput.manivel.ea:quantum*1} that $X_d$ is a
  cominuscule variety, and using
  \cite[Fact~20]{chaput.manivel.ea:quantum*1} we obtain $d_{X_d}(3) =
  d$.  It now follows from Proposition~\ref{prop:dx3xd} applied to
  $X_d$ that $z \in \Gamma_{d-d_0}(\Gamma_{d_0}(x,y))$.  Finally, if
  $d_X(2) < d < d_X(3)$, then $X$ is a Grassmannian of type A or the
  Cayley plane $E_6/P_6$.  We consider these cases in turn.

  If $X = \Gr(m,n)$ is a Grassmannian of type A, then
  \cite[Lemma~1]{buch:quantum} implies $x,y,z \in X' := \Gr(m-a,B/A)$
  for some subspaces $A \subset B \subset \C^n$ such that $a :=
  \dim(A) \geq m-d$ and $\dim(B) \leq m+d$.  Since $d_{X'}(3) \leq d$
  we deduce that $z \in \Gamma_{d-d_0}(\Gamma_{d_0}(x,y))$ by applying
  Proposition~\ref{prop:dx3xd} to $X'$.

  Finally, assume that $X = E_6/P_6$ is the Cayley plane, in which
  case we have $d = 3$.  Since $\Gamma_2(X_1) = X$, we may also assume
  that $d_0 = 2$ and $\Gamma_2(x,y) = X_2$.  With the notation from
  the proof of Proposition~\ref{prop:dx3xd} we have $X_2 =
  X(s_6s_5s_4s_2s_3s_4s_5s_6)$, the only non-singular Schubert variety
  of dimension 8, and Lemma~\ref{lem:deg1nbhd} implies that
  $\Gamma_1(X_2)$ is the unique Schubert divisor in $X$.  Since
  $d_X(3)=4$, we also have $\Gamma_1(X_2) \subset \Gamma_3(x,y)
  \subsetneq X$.  We conclude that $\Gamma_1(X_2) = \Gamma_3(x,y)$ as
  $\Gamma_3(x,y)$ is irreducible.
\end{proof}

\begin{cor}\label{cor:Zd}
  Let $\bd = (d_0,d_1,\dots,d_r)$ be a sequence with $d_i>0$ for
  $i>0$.  Then $\cZ_\bd = \{ (x,y,z) \in \cZ_{d_0,2} \times X \mid z
  \in \Gamma_{|\bd|}(x,y) \}$.
\end{cor}
\begin{proof}
  By definition we have $\cZ_\bd = \{ (x,y,z) \in X^3 \mid z \in
  \Gamma_\bd(x,y) \}$.  The corollary is true because $\Gamma_\bd(x,y)
  \neq \emptyset$ if and only if $(x,y) \in \cZ_{d_0,2}$, in which
  case we have $\Gamma_\bd(x,y) = \Gamma_{|\bd|}(x,y)$ by
  Corollary~\ref{cor:gammaschubert} and Lemma~\ref{lem:linechain}.
\end{proof}

We need the following rationality property of Gromov-Witten varieties
defined by triples of points, which in some cases are in special
position \cite[Thm.~0.2]{chaput.perrin:rationality}.  Notice that when
$X$ is a Grassmannian of type A, this follows from
\cite[Thm.~2.1]{buch.mihalcea:quantum}.

\begin{thm}[\cite{chaput.perrin:rationality}]\label{thm:rat3gw}
  Let $X = G/P$ be a cominuscule variety and let $d \in \N$.  Then the
  variety $M_{d,3}(x,y,z) \subset \Mb_{0,3}(X,d)$ is rational for all
  points $(x,y,z)$ in a dense open subset of $\cZ_{d,3} \subset X^3$.
\end{thm}

\begin{thm}\label{thm:rat3bdgw}
  Let $X = G/P$ be cominuscule and $\bd = (d_0,\dots,d_r)$ any
  sequence with $d_i>0$ for $i>0$.  Then $M_{\bd}(x,y,z)$ is
  rationally connected for all points $(x,y,z)$ in a dense open subset
  of $\cZ_\bd \subset X^3$.
\end{thm}
\begin{proof}
  The result follows from Theorem~\ref{thm:rat3gw} when $r=0$, so
  assume that $r>0$.  By induction on $r$ we may assume that
  $M_{\bd'}(x,y,t)$ is rationally connected for all points $(x,y,t)$
  in a dense open subset of $\cZ_{\bd'}$, where $\bd' =
  (d_0,\dots,d_{r-1})$.  Let $(x,y) \in \cZ^*_{d_0,2}$ and set $\Omega
  = \Gamma_{\bd'}(x,y)$.  Since $G$ acts transitively on
  $\cZ^*_{d_0,2}$ by Proposition~\ref{prop:cmpfacts}(a), there exists
  a dense open subset $\Omega^* \subset \Omega$ such that
  $M_{\bd'}(x,y,t)$ is rationally connected for all $t \in \Omega^*$.
  By Lemma~\ref{lem:ratint2} there exists a dense open subset
  $\Gamma^* \subset \Gamma_\bd(x,y) = \Gamma_{d_r}(\Omega)$ such that
  $\Omega^* \cap \Gamma^*_{d_r}(z) \neq \emptyset$ and $\Omega \cap
  \Gamma_{d_r}(z)$ is unirational and normal for all $z \in \Gamma^*$.

  By replacing $\Gamma^*$ with a smaller set, we may also assume that
  $M_\bd(x,y,z)$ is locally irreducible for all $z \in \Gamma^*$.
  Indeed, since $M_\bd(x)$ is unirational by
  Corollary~\ref{cor:unirat}, and Proposition~\ref{prop:rat2gw}(b)
  implies that $\ev_2 : M_\bd(x) \to \Gamma_{d_0}(x)$ is locally
  trivial over $\Gamma^*_{d_0}(x)$, it follows that $M_\bd(x,y)$ is
  unirational.  Since $\Gamma_\bd(x,y)$ is a Schubert variety by
  Corollary~\ref{cor:gammaschubert} and $\ev_3 : M_\bd(x,y) \to
  \Gamma_\bd(x,y)$ is surjective, the Kleiman-Bertini theorem
  \cite[Rmk.~7]{kleiman:transversality} applied to the Borel-action on
  the open cell $\Gamma_\bd(x,y)^\circ$ shows that $M_\bd(x,y,z)$ is
  locally irreducible for all points $z$ in a dense open subset of
  $\Gamma_\bd(x,y)$.

  We claim that $M_\bd(x,y,z)$ is rationally connected for all $z \in
  \Gamma^*$.  The space $M_\bd$ is the product of the maps $\ev_3 :
  M_{\bd'} \to X$ and $\ev_1 : M_{d_r,2} \to X$.  Let $f : M_\bd \to
  X$ be the morphism defined by the product.  This map restricts to a
  surjective morphism $M_\bd(x,y,z) \to \Omega \cap \Gamma_{d_r}(z)$,
  whose fibers are given by $f^{-1}(t) \cap M_\bd(x,y,z) =
  M_{\bd'}(x,y,t) \times M_{d_r,2}(t,z)$.  Since $M_{d_r,2}(z)$ is
  unirational and Proposition~\ref{prop:rat2gw}(b) implies that the
  map $\ev_2 : M_{d_r,2}(z) \to \Gamma_{d_r}(z)$ is locally trivial
  over $\Gamma^*_{d_r}(z)$, it follows that $M_{d_r,2}(t,z)$ is
  unirational for all $t \in \Gamma^*_{d_r}(z)$.  We deduce that
  $f^{-1}(t) \cap M_\bd(x,y,z)$ is rationally connected for all $t \in
  \Omega^* \cap \Gamma^*_{d_r}(z)$.  By using the Stein factorization
  of the map $M_\bd(x,y,z) \to \Omega \cap \Gamma_{d_r}(z)$ and the
  fact that $\Omega \cap \Gamma_{d_r}(z)$ is normal, it follows from
  Zariski's main theorem \cite[III.11.4]{hartshorne:algebraic*1} that
  $M_\bd(x,y,z)$ is connected and therefore irreducible.  The claim
  now follows from Theorem~\ref{thm:ratconn}.
  
  Define a morphism $\rho : G \times \Gamma_\bd(x,y) \to \cZ_\bd$ by
  $\rho(g,z) = (g.x, g.y, g.z)$.  It follows from
  Proposition~\ref{prop:cmpfacts}(a) that the image of $\rho$ contains
  $\cZ_\bd \cap (\cZ^*_{d_0,2} \times X)$, which is a dense open
  subset of $\cZ_\bd$.  This implies that $\rho(G \times \Gamma^*)$
  contains a dense open subset of $\cZ_\bd$, which completes the proof
  of the theorem.
\end{proof}

\section{Quantum $K$-theory of cominuscule
  varieties}\label{sec:qkcomin}

Let $X = G/P$ be a cominuscule variety and let $K(X)$ denote its
Grothendieck ring.  A short summary of the properties of this ring can
be found in \cite[\S 3]{buch.ravikumar:pieri}, while many more details
can be found in \cite{brion:lectures}.  Each element $w \in W^P$
defines a Schubert class $\cO_w = [\cO_{X(w_0 w)}] \in K(X)$, where
$w_0 \in W$ is the longest element, and these classes form a
$\Z$-basis for $K(X)$.  The dual Schubert classes $\cO_w^\vee \in
K(X)$ are defined by $\euler{X}(\cO_u \cdot \cO_v^\vee) =
\delta_{u,v}$ for $u,v \in W^P$, where $\euler{X} : K(X) \to \Z$ is
the sheaf Euler characteristic map.

Given classes $\alpha_1, \dots, \alpha_n \in K(X)$, we set $\alpha_1
\otimes \dots \otimes \alpha_n = \prod_{i=1}^n \pi_i^*(\alpha_i) \in
K(X^n)$, where $\pi_i : X^n \to X$ is the $i$-th projection.  Together
with a degree $d \in \N$, these classes define the $K$-theoretic
Gromov-Witten invariant
\[
I_d(\alpha_1,\dots,\alpha_n) = \euler{M_{d,n}}(\ev^*(\alpha_1 \otimes
\dots \otimes \alpha_n)) \,.
\]
The quantum $K$-theory ring of $X$ is an algebra over $\Z[[q]]$, which
as a $\Z[[q]]$-module is given by $\QK(X) = K(X) \otimes_\Z \Z[[q]]$.
The multiplicative structure is defined by
\[
\cO_u \star \cO_v = \sum_{w,d} N^{w,d}_{u,v}\, q^d\, \cO_w \,,
\]
where the sum is over all $w \in W^P$ and $d \in \N$.  The structure
constants $N^{w,d}_{u,v}$ are defined by
\[
N^{w,d}_{u,v} = \sum_{\bd=(d_0,\dots,d_r),\kappa_1,\dots,\kappa_r}
(-1)^r\, I_{d_0}(\cO_u, \cO_v, \cO_{\kappa_1}^\vee)\,
\prod_{i=1}^r I_{d_i}(\cO_{\kappa_i}, \cO_{\kappa_{i+1}}^\vee) \,,
\]
the sum over all sequences $\bd = (d_0,\dots,d_r)$ with $|\bd|=d$ and
$d_i>0$ for $i>0$, and all elements $\kappa_1,\dots,\kappa_r \in W^P$.
Notice that the sign $(-1)^r$ and the number of elements $\kappa_i$
depend on the length of $\bd$, and we write $\kappa_{r+1} = w$.  A
theorem of Givental \cite{givental:on} states that $\QK(X)$ is an
associative ring.

In this section we prove that any product of Schubert classes in
$\QK(X)$ has only finitely many non-zero terms.  We start by observing
that each structure constant $N^{w,d}_{u,v}$ can also be expressed as
an alternating sum of Euler characteristics computed on the spaces
$M_\bd$.  The following lemma generalizes to any homogeneous space
with the same proof.

\begin{lemma}\label{lem:bdgw}
  Let $u,v,w \in W^P$ and let $\bd = (d_0,d_1,\dots,d_r)$ be any
  sequence such that $d_i>0$ for $i>0$.  Then
  \[
  \euler{M_\bd}(\ev^*(\cO_u \otimes \cO_v \otimes \cO_w^\vee)) =
  \sum_{\kappa_1,\dots,\kappa_r}
  I_{d_0}(\cO_u,\cO_v,\cO_{\kappa_1}^\vee) \prod_{i=1}^r
  I_{d_i}(\cO_{\kappa_i},\cO_{\kappa_{i+1}}^\vee)
  \]
  where the sum is over all $\kappa_1,\dots,\kappa_r \in W^P$ and we
  set $\kappa_{r+1} = w$.
\end{lemma}
\begin{proof}
  We may assume that $r>0$.  Set $\bd' = (d_0,\dots,d_{r-1})$.  It is
  enough to show that
  \begin{equation}\label{eqn:dsum_enough}
  \euler{M_\bd}(\ev^*(\cO_u \otimes \cO_v \otimes \cO_w^\vee)) =
  \sum_{\kappa \in W^P}
  \euler{M_{\bd'}}(\ev^*(\cO_u \otimes \cO_v \otimes \cO_\kappa^\vee)) 
  \cdot I_{d_r}(\cO_\kappa, \cO_w^\vee) \,.
  \end{equation}
  Let $\Delta : X \to X^2$ be the diagonal embedding.  The projection
  formula implies that $\euler{X^2}(\Delta_*[\cO_X] \cdot \cO_\sigma
  \otimes \cO_\tau^\vee) = \euler{X}(\cO_\sigma \cdot \cO_\tau^\vee) =
  \delta_{\sigma,\tau}$ for all $\sigma,\tau \in W^P$, and the class
  $\Delta_*[\cO_X] \in K(X^2)$ is uniquely determined by this
  property.  We deduce that (cf.\
  \cite[Thm.~3.4.1(i)]{brion:lectures})
  \[
  \Delta_*([\cO_X]) = \sum_{\kappa\in W^P} \cO_\kappa^\vee \otimes
  \cO_\kappa \in K(X^2) \,.
  \]
  Since the horizontal maps are flat in the fiber square
  \[
  \xymatrix{ M_{\bd'} \times M_{d_r} \ar[rr]^{\ev_3 \times \ev_1} &&
    X^2 \\
    M_\bd \ar[rr] \ar[u]_{\Delta'} && X \ar[u]_{\Delta}}
  \]
  we obtain $\Delta'_*[\cO_{M_\bd}] = (\ev_3 \times
  \ev_1)^*\Delta_*[\cO_X] = \sum_\kappa \ev_3^*(\cO_\kappa^\vee)
  \otimes \ev_1^*(\cO_\kappa)$.  Equation (\ref{eqn:dsum_enough})
  follows from this by another application of the projection formula.
\end{proof}

We need the Gysin formula from \cite[Thm.~3.1]{buch.mihalcea:quantum}
stated as Proposition~\ref{prop:gysin} below.  Notice that the
statement in \cite{buch.mihalcea:quantum} requires that the general
fibers of $f$ are rational.  However, this was used only to conclude
that the structure sheaf of any general smooth fiber has vanishing
higher cohomology, and rational connectivity suffices for this, see
e.g.\ \cite[Cor.~4.18(a)]{debarre:higher-dimensional}.

\begin{prop}[\cite{buch.mihalcea:quantum}]\label{prop:gysin}
  Let $f : X \to Y$ be a surjective morphism of projective varieties
  with rational singularities.  If the general fibers of $f$ are
  rationally connected, then $f_*[\cO_X] = [\cO_Y] \in K(Y)$.
\end{prop}

\ignore{
\begin{lemma}\label{lem:ratsingZd}
  Let $\bd = (d_0,\dots,d_r)$ be such that $d_i > 0$ for $i > 0$.
  Assume that $|\bd| \geq d_X(3)$.  Then $\cZ_\bd$ has rational
  singularities and $\ev_* [\cO_{M_\bd}] = [\cO_{\cZ_\bd}]$.
\end{lemma}
\begin{proof}
  It follows from Corollary~\ref{cor:Zd} that $\cZ_\bd = \cZ_{d_0,2}
  \times X$, so Lemma~\ref{lem:ratint2}(a) implies $\cZ_\bd$ has
  rational singularities.  Since also $M_\bd$ has rational
  singularities by Proposition~\ref{prop:ratsingMde}, the identity
  $\ev_* [\cO_{M_\bd}] = [\cO_{\cZ_\bd}]$ follows from
  Proposition~\ref{prop:gysin} and Theorem~\ref{thm:rat3bdgw}.
\end{proof}
}

\begin{proof}[Proof of Theorem~\ref{thm:intro_vanish}]
  For each sequence $\bd = (d_0,\dots,d_r)$ with $d_i>0$ for $i>0$ we
  choose a resolution of singularities $\pi : \wt\cZ_\bd \to \cZ_\bd$.
  By Corollary~\ref{cor:Zd} we may assume that this resolution depends
  only on $|\bd|$ and $\min(d_0,d_X(2))$.  Then form the following
  commutative diagram, where $M'_\bd \subset \wt\cZ_\bd
  \times_{\cZ_\bd} M_\bd$ is the irreducible component mapping
  birationally to $M_\bd$, and $\wt M_\bd$ is a resolution of
  singularities of this component.
  \[
  \xymatrix{
    \wt M_\bd \ar[r]^\varphi & M'_\bd \ar[r]^{\subset\ \ \ \ \ \ } & 
    \wt\cZ_\bd \times_{\cZ_\bd} M_\bd \ar[r]^{\ \ \ \ \pi'} \ar[d]^{\ev'} &
    M_\bd \ar[d]^\ev \\
    && \wt\cZ_\bd \ar[r]^\pi & \cZ_\bd \ar[r]^\subset & X^3
  }
  \]

  It follows from Zariski's main theorem that the fibers of $\pi'
  \varphi$ are connected.  Using Theorem~\ref{thm:rat3bdgw} we deduce
  that the general fibers of $\ev'\varphi$ are connected.  Since the
  map $\ev'\varphi$ is smooth over a dense open subset of $\wt\cZ_\bd$
  by \cite[III.10.7]{hartshorne:algebraic*1}, it follows that the
  general fibers of $\ev'\varphi$ are in fact irreducible.
  Theorem~\ref{thm:rat3bdgw} therefore shows that the general fibers
  of $\ev'\varphi$ are rationally connected, so we obtain $(\ev'
  \varphi)_* [\cO_{\wt M_\bd}] = [\cO_{\wt\cZ_\bd}]$ by
  Proposition~\ref{prop:gysin}.  Since $M_\bd$ has rational
  singularities, we also have $(\pi' \varphi)_* [\cO_{\wt M_\bd}] =
  [\cO_{M_\bd}]$.  We deduce from the projection formula that
  \[
  \euler{M_\bd}(\ev^*(\cO_u \otimes
  \cO_v \otimes \cO_w^\vee)) = \euler{\wt\cZ_\bd}(\pi^*(\cO_u \otimes
  \cO_v \otimes \cO_w^\vee)) \,.
  \]
  Now Lemma~\ref{lem:bdgw} implies that
  \[
  \begin{split}
    N^{w,d}_{u,v} 
    &= \sum_\bd (-1)^r\, 
    \euler{M_\bd}(\ev^*(\cO_u\otimes\cO_v\otimes\cO_w^\vee)) \\
    &= \sum_\bd (-1)^r\, 
\euler{\wt\cZ_\bd}(\pi^*(\cO_u \otimes
  \cO_v \otimes \cO_w^\vee)) 
  \end{split}
  \]
  where both sums are over all sequences $\bd = (d_0,\dots,d_r)$ with
  $d_i>0$ for $i>0$ and $|\bd|=d$, and the sign $(-1)^r$ depends on
  the length of $\bd$.  Notice that the terms of the second sum depend
  only on $\min(d_0,d_X(2))$ and $r$.  In particular, the
  contributions of the sequences $\bd=(d)$ and $\bd=(d-1,1)$ cancel
  each other out.  Now let $0 \leq d' \leq d-2$.  For each $r$ with $1
  \leq r \leq d-d'$, there are exactly $\binom{d-d'-1}{r-1}$ sequences
  $\bd$ in the sum for which $d_0=d'$ and the length of $\bd$ is
  $r+1$.  Since $\sum_{r=1}^{d-d'} (-1)^r \binom{d-d'-1}{r-1} = 0$, it
  follows that the corresponding terms cancel each other out.  This
  completes the proof.
\end{proof}

\begin{remark}\label{rmk:eqkfinite}
  Theorem~\ref{thm:intro_vanish} is true also for the structure
  constants of the equivariant quantum $K$-theory ring $\QK_T(X)$,
  with the same proof.  In fact, if $\wt\cZ_\bd$ and $\wt M_\bd$ are
  chosen to be $T$-equivariant resolutions
  \cite[Thm.~7.6.1]{villamayor-u:patching}, then all maps used in the
  proofs of Lemma~\ref{lem:bdgw} and Theorem~\ref{thm:intro_vanish}
  are equivariant, and the arguments go through without change.  More
  details about the ring $\QK_T(X)$ can be found in
  \cite{buch.mihalcea:quantum}.
\end{remark}


\begin{thebibliography}{10}

\bibitem{brion:positivity}
M.~Brion, \emph{Positivity in the {G}rothendieck group of complex flag
  varieties}, J. Algebra \textbf{258} (2002), no.~1, 137--159, Special issue in
  celebration of Claudio Procesi's 60th birthday. \MR{MR1958901 (2003m:14017)}

\bibitem{brion:lectures}
\bysame, \emph{Lectures on the geometry of flag varieties}, Topics in
  cohomological studies of algebraic varieties, Trends Math., Birkh\"auser,
  Basel, 2005, pp.~33--85. \MR{MR2143072 (2006f:14058)}

\bibitem{buch:quantum}
A.~S. Buch, \emph{Quantum cohomology of {G}rassmannians}, Compositio Math.
  \textbf{137} (2003), no.~2, 227--235. \MR{MR1985005 (2004c:14105)}

\bibitem{buch.kresch.ea:gromov-witten}
A.~S. Buch, A.~Kresch, and H.~Tamvakis, \emph{Gromov-{W}itten invariants on
  {G}rassmannians}, J. Amer. Math. Soc. \textbf{16} (2003), no.~4, 901--915
  (electronic). \MR{MR1992829 (2004h:14060)}

\bibitem{buch.mihalcea:quantum}
A.~S. Buch and L.~C. Mihalcea, \emph{Quantum {$K$}-theory of {G}rassmannians},
  to appear in Duke Math.\ J., arXiv:0708.3418, 2008.

\bibitem{buch.ravikumar:pieri}
A.~S. Buch and V.~Ravikumar, \emph{Pieri rules for the ${K}$-theory of
  cominuscule {G}rassmannians}, arXiv:1005.2605.

\bibitem{chaput.manivel.ea:quantum*1}
P.-E. Chaput, L.~Manivel, and N.~Perrin, \emph{Quantum cohomology of minuscule
  homogeneous spaces}, Transform. Groups \textbf{13} (2008), no.~1, 47--89.
  \MR{MR2421317}

\bibitem{chaput.perrin:on}
P.-E. Chaput and N.~Perrin, \emph{On the quantum cohomology of adjoint
  varieties}, to appear in Proc. of the London Math. Soc.,
  ar$\chi$iv:0904:4824v1.

\bibitem{chaput.perrin:rationality}
\bysame, \emph{Rationality of some {G}romov-{W}itten varieties and application
  to quantum {$K$}-theory}, to appear in Comm. in Contemp. Math.,
  ar$\chi$iv:0905.4394.

\bibitem{jong.starr.ea:families}
J.~De~Jong, J.~Starr, and Xuhua He, \emph{Families of rationally simply
  connected varieties over surfaces and torsors for semisimple groups},
  ar$\chi$iv:0809.5224.

\bibitem{debarre:higher-dimensional}
O.~Debarre, \emph{Higher-dimensional algebraic geometry}, Universitext,
  Springer-Verlag, New York, 2001. \MR{1841091 (2002g:14001)}

\bibitem{fulton.pandharipande:notes}
W.~Fulton and R.~Pandharipande, \emph{Notes on stable maps and quantum
  cohomology}, Algebraic geometry---{S}anta {C}ruz 1995, Proc. Sympos. Pure
  Math., vol.~62, Amer. Math. Soc., Providence, RI, 1997, pp.~45--96.
  \MR{MR1492534 (98m:14025)}

\bibitem{fulton.woodward:on}
W.~Fulton and C.~Woodward, \emph{On the quantum product of {S}chubert classes},
  J. Algebraic Geom. \textbf{13} (2004), no.~4, 641--661. \MR{MR2072765
  (2005d:14078)}

\bibitem{givental:on}
A.~Givental, \emph{On the {WDVV} equation in quantum {$K$}-theory}, Michigan
  Math. J. \textbf{48} (2000), 295--304, Dedicated to William Fulton on the
  occasion of his 60th birthday. \MR{MR1786492 (2001m:14078)}

\bibitem{graber.harris.ea:families}
T.~Graber, J.~Harris, and J.~Starr, \emph{Families of rationally connected
  varieties}, J. Amer. Math. Soc. \textbf{16} (2003), no.~1, 57--67
  (electronic). \MR{MR1937199 (2003m:14081)}

\bibitem{hartshorne:algebraic*1}
R.~Hartshorne, \emph{Algebraic geometry}, Springer-Verlag, New York, 1977,
  Graduate Texts in Mathematics, No. 52. \MR{MR0463157 (57 \#3116)}

\bibitem{kim.pandharipande:connectedness}
B.~Kim and R.~Pandharipande, \emph{The connectedness of the moduli space of
  maps to homogeneous spaces}, Symplectic geometry and mirror symmetry
  ({S}eoul, 2000), World Sci. Publ., River Edge, NJ, 2001, pp.~187--201.
  \MR{MR1882330 (2002k:14021)}

\bibitem{kleiman:transversality}
S.~L. Kleiman, \emph{The transversality of a general translate}, Compositio
  Math. \textbf{28} (1974), 287--297. \MR{MR0360616 (50 \#13063)}

\bibitem{popov:generically}
V.~L. Popov, \emph{Generically multiple transitive algebraic group actions},
  Algebraic groups and homogeneous spaces, Tata Inst. Fund. Res. Stud. Math.,
  Tata Inst. Fund. Res., Mumbai, 2007, pp.~481--523. \MR{2348915 (2008i:14073)}

\bibitem{springer:linear}
T.~A. Springer, \emph{Linear algebraic groups}, Algebraic geometry, 4
  ({R}ussian), Itogi Nauki i Tekhniki, Akad. Nauk SSSR Vsesoyuz. Inst. Nauchn.
  i Tekhn. Inform., Moscow, 1989, Translated from the English, pp.~5--136,
  310--314, 315. \MR{MR1100484 (92g:20061)}

\bibitem{villamayor-u:patching}
O.~E. Villamayor~U., \emph{Patching local uniformizations}, Ann. Sci. \'Ecole
  Norm. Sup. (4) \textbf{25} (1992), no.~6, 629--677. \MR{MR1198092
  (93m:14012)}

\bibitem{zak:tangents}
F.~L. Zak, \emph{Tangents and secants of algebraic varieties}, Translations of
  Mathematical Monographs, vol. 127, American Mathematical Society, Providence,
  RI, 1993, Translated from the Russian manuscript by the author. \MR{1234494
  (94i:14053)}

\end{thebibliography}

\providecommand{\bysame}{\leavevmode\hbox to3em{\hrulefill}\thinspace}
\providecommand{\MR}{\relax\ifhmode\unskip\space\fi MR }
\providecommand{\MRhref}[2]{%
  \href{http://www.ams.org/mathscinet-getitem?mr=#1}{#2}
}
\providecommand{\href}[2]{#2}

\end{document}